\documentclass[12pt,leqno,twoside]{amsart}

\usepackage[margin=1in]{geometry}

\usepackage{amsmath,amscd,amssymb,amsfonts,latexsym,wasysym, mathrsfs, mathtools,hhline,color}
\usepackage[all, cmtip]{xy}
\usepackage{csquotes}
\usepackage{url}
\usepackage{comment,l3keys2e, array, pgfcore}
\usepackage{ulem}
\usepackage{nicematrix}

\definecolor{hot}{RGB}{65,105,225}

\usepackage[pagebackref=true,colorlinks=true, linkcolor=hot ,  citecolor=hot, urlcolor=hot]{hyperref}
\usepackage{ textcomp }
\usepackage{ tipa }
\usepackage{graphicx,enumerate}

\usepackage{geometry}
\theoremstyle{plain}
\newtheorem{theorem}{Theorem}[section]
\newtheorem{prop}[theorem]{Proposition}

\newtheorem{lm}[theorem]{Lemma}

\newtheorem{cor}[theorem]{Corollary}

\newtheorem{thrm}[theorem]{Theorem}

\theoremstyle{definition}

\newtheorem{defn}[theorem]{Definition}

\newtheorem{rmk}[theorem]{Remark}
\newtheorem{remark}[theorem]{Remark}

\newtheorem{ex}[theorem]{Example}
\newtheorem*{ex*}{Example}

\def\be{\begin{equation}}
\def\ee{\end{equation}}

\def\bt{\begin{thrm}}
\def\et{\end{thrm}}

\def\bc{\begin{cor}}
\def\ec{\end{cor}}

\def\br{\begin{rmk}}
\def\er{\end{rmk}}

\def\bp{\begin{prop}}
\def\ep{\end{prop}}

\def\bl{\begin{lm}}
\def\el{\end{lm}}

\def\bex{\begin{ex}}
\def\eex{\end{ex}}

\def\bd{\begin{defn}}
\def\ed{\end{defn}}

\newcommand{\C}{\mathbb{C}}

\newcommand{\R}{\mathbb{R}}

\newcommand{\Z}{\mathbb{Z}}

\newcommand{\Q}{\mathbb{Q}}
\newcommand{\K}{\mathbb{K}}

\newcommand{\F}{\mathbb{F}}

\newcommand{\Tors}{\mathrm{Tors}}

\newcommand{\sA}{\mathcal{A}}

\newcommand{\sL}{\mathcal{L}}

\newcommand{\coker}{\mathrm{coker}}

\begin{document}

\title[]{Spectral sequences, Massey products \\ and  homology of covering spaces}



\author{Yongqiang Liu}
\address{Y. Liu: The Institute of Geometry and Physics, University of Science and Technology of China, 96 Jinzhai Road, Hefei, Anhui 230026 China}
\email{liuyq@ustc.edu.cn}

\author{Laurentiu Maxim}
\address{L. Maxim: Department of Mathematics,         University of Wisconsin-Madison,  480 Lincoln Drive, Madison WI 53706-1388, USA,\newline
{\text and} \newline Institute of Mathematics of the Romanian Academy, P.O. Box 1-764, 70700 Bucharest, ROMANIA}
\email {maxim@math.wisc.edu}

\author{Botong Wang}
\address{B. Wang: Department of Mathematics,         University of Wisconsin-Madison,  480 Lincoln Drive, Madison WI 53706-1388, USA.}
\email {wang@math.wisc.edu}


\date{\today}



\begin{abstract}
We revisit the equivariant spectral sequence considered by Papadima-Suciu, and show that all its differentials are computed by higher order Massey products. As a first application, we extend to arbitrary field coefficients results of Pajitnov relating the size of Jordan blocks for the eigenvalue $1$ part of the Alexander modules to the length of nonvanishing Massey products in cohomology. We also give computable upper bounds for the mod $p$ Betti numbers of prime power cyclic covers, and resp. for the ranks of the cohomology groups with coefficients in a prime order rank one local system. Under suitable conditions, these bounds are improvements of the ones obtained by Papadima-Suciu. We also specialize these results to the case of hyperplane arrangement complements, showing e.g., that vanishing of higher-order Massey products implies that the mod $p$ Betti numbers of prime $p$ tower cyclic covers are combinatorially determined. 
\end{abstract}

\maketitle


\section{Introduction}

\subsection{Infinite cyclic cover}

Let $X$ be a connected finite CW complex with a group epimorphism $\nu \colon \pi_1(X)\twoheadrightarrow \Z$, and consider the corresponding infinite cyclic cover $X^{\nu}$.  Fix a ground field $\K$, and set $R=\K[\Z]\simeq \K[t^{\pm 1}]$.  Then the {\it $i$-th Alexander module} $H_i(X^\nu, \K)$ of $(X,\nu)$ is a finitely generated $R$-module for any integer $i$. 

If we assume that $\nu$ is induced by a fibration $f\colon X\to S^1$ with fiber $F$, there are many results in the literature that relate  the size of the Jordan blocks for the monodromy action on $H_*(F,\C)\cong H_*(X^\nu,\C)$ to the Massey products of $X$.
For instance, Fern\'andez--Gray--Morgan  \cite{FGM} used the relation between non-vanishing Massey products of length $2$ and the Jordan blocks
of size greater than $1$ to prove that certain mapping tori do not admit the structure of a K\"ahler manifold.
Papadima--Suciu  \cite{PS10B} showed that if $\pi_1(X)$ is $1$-formal,  the eigenvalue $1$ part of $H_1(F,\C)$ is semi-simple, see also \cite{PS10}. Moreover, Bazzoni--Fern\'andez--Mu\~noz \cite{BFM} showed that the existence of Jordan blocks of size $2$ implies the existence of non-zero triple Massey products.

Without the circle fibration assumption, the size of Jordan blocks for the eigenvalue $1$ part of $H_*(X^\nu,\C)$ was studied by Papadima--Suciu in \cite[Section 9]{PS10}. Also, Budur and the first and third authors of this paper studied this question for compact K\"ahler manifolds and smooth complex quasi-projective varieties, see \cite{BLW}. In particular, the torsion part of $H_*(X^\nu,\C)$ is semi-simple for $X$ a compact K\"ahler manifold, since such a manifold is formal over $\C$, cf. \cite{DGMS}. Finally, Pajitnov identified the length of certain non-zero higher order Massey products with the size of Jordan
blocks for the eigenvalue $1$ part of $H_*(X^\nu,\C)$, see \cite[Theorem 5.1]{Paj17}, and also \cite{Paj19} for the twisted version.

One of the motivations for this paper is to extend Pajitnov's results  \cite{Paj17} to positive characteristic field coefficients. In this context, formality is often not available, e.g., compact K\"ahler manifolds \cite{Eke} and hyperplane arrangement complements \cite{Mat} are in general not formal over $\F_p$. 

There exist several spectral sequences in the literature, designed to study the (torsion submodule of the) Alexander module $H_*(X^\nu,\K)$. Here we use 
a Massey type spectral sequence, which as we show in Proposition \ref{prop key} is dual to the $J$-adic equivariant spectral sequence studied by Papadima--Suciu \cite{PS10}. Furthermore, standard facts about spectral sequences can be used to show that the differentials of the Massey type spectral sequence are computed in terms of a special type of higher order Massey products (see Definition \ref{higherM}  and see also \cite{KP,Paj17,Paj19} for similar statements).
The latter identification leads us to the following generalization of the above mentioned result of Pajitnov  \cite[Theorem 5.1]{Paj17} for arbitrary field coefficients:   
\begin{theorem} \label{thm main}
For $i\geq 0$, the maximal size of Jordan blocks for the eigenvalue $1$ part of $H_i(X^\nu,\K)$ is one less than the length of the highest nonvanishing  Massey product {on degree $i$} associated to a $1$-cocycle representative $\eta$ for $\nu \in H^1(X^\nu,\K)$.
    In particular, if all higher order Massey products of $X$ are trivial, the eigenvalue $1$ part of $H_*(X^\nu,\K)$ is semi-simple.
\end{theorem}

As a byproduct, we also get the following application, extending some results from \cite{Memoir} and \cite{BR}. 
\bc \label{cor singular}  Let $X$ be a connected complex $n$-dimensional algebraic variety, possibly singular. Assume that $W_0 H^1(X,\C)=0$, where $W$ is the weight filtration. Then for any epimorphism $\nu \colon \pi_1(X)\twoheadrightarrow \Z$, the size of Jordan blocks for the eigenvalue $1$ part of $H_i(X^\nu,\C)$ is at most $ \min\{2i+2,2n\}$. For $H_1(X^\nu,\C)$, the bound can be improved to be  $ 3$. 
If we further assume that $W_1 H^1(X,\C)=0$, then the size of Jordan blocks for the eigenvalue $1$ part of $H_i(X^\nu,\C)$ is at most $ \min\{i+1,n\}$. For $H_1(X^\nu,\C)$, the bound can be improved to be $1$. \ec

\subsection{Prime tower cyclic cover}
If $p$ is a prime integer and $r$ is a positive integer, by further projecting the epimorphism $\nu \colon \pi_1(X)\twoheadrightarrow \Z$ to $\Z_{p^r}$ yields a corresponding finite cyclic $p^r$-fold cover $X_r$. 
Let $\K=\F_p$, and let $\eta_p \in C^1(X,\F_p)$ be the corresponding $1$-cocycle (where the subscript is used to emphasize the characteristic of the field). Then by the $J$-adic filtration, one gets a truncated spectral sequence induced from the Massey type spectral sequence discussed above, which computes the mod $p$ homology $H_*(X_r,\F_p)$ of the cover $X_r$, and which was considered, e.g.,  by Papadima--Suciu in \cite[Section 7.1]{PS10}. When $r=1$, this spectral sequence specializes to the one studied by Reznikov in  \cite{Rez}.

By our key technical result of Proposition \ref{prop key}, this truncated spectral sequence degenerates at most at the $E_{p^r}$-page. Moreover,  the differential map on the $E_1$-page  
is given by the left cup product with $[\eta_p]$, 
and  the $E_k$-page for $k\geq 2$ is computed by the higer order Massey product associated to $\eta_p$. This spectral sequence is used in Section \ref{sec4} to obtain upper bounds for the Betti numbers of $X_r$ with $\F_p$-coefficients. More precisely, 
denoting by $$\beta_i(X,\eta_p):=\dim_{\F_p} H^i(H^*(X,\F_p), [\eta_p]\cup -)$$ the corresponding $i$-th Aomoto Betti number with $\F_p$-coefficients, we have the following.

\bp \label{prop prime tower cover}
Let 
$X_r\to X$ be the $p^r$-fold cover defined above. 
Then for 
any $i\geq 0$ we have
\begin{equation}\label{iinequality}
   b_i(X_r,\F_p) \leq b_i(X,\F_p) + (p^r-1)\cdot \beta_i(X,\eta_p).
\end{equation}
\ep

 If we further assume that $r=1$ and  $p=2$, the spectral sequence degenerates at {the} $E_2$-page and one gets a long exact sequence, also known as the transfer long exact sequence
$$ \cdots \to H^{i-1}(X,\F_2)  \xrightarrow{[\eta_2] \cup} H^{i}(X,\F_2) \to H^i(Y,\F_2) \to H^{i}(X,\F_2)  \xrightarrow{[\eta_2] \cup} H^{i+1}(X,\F_2)\to \cdots,$$
where $Y$ is the associated double cover of $X$. 
Yoshinaga \cite{Yos} used this long exact sequence to show that 
\begin{equation} \label{mod 2 equation}
    b_i(Y, \F_2)= b_i(X,\F_2)+\beta_i(X,\eta_2).
\end{equation}
In particular, $b_i(Y,\F_2)$ is determined by combinatorics when $X$ is a hyperplane arrangement complement.

On the other hand, in Section \ref{sec4} we show the following results for $p^r>2$. 
\begin{prop}\label{cor2i} Let $X_r$ be the $p^r$-fold cover of $X$ associated to the reduction mod $\Z_{p^r}$ of the epimorphism $\nu\colon \pi_1(X)\twoheadrightarrow \Z$.
    For $p^r>2$,  (\ref{iinequality}) holds as an equality for all $i$ if and only if the above spectral sequence degenerates at the $E_2$-page. In particular, (\ref{iinequality}) holds as a strict inequality at degree $i$ if there exist a non-trivial $k$-tuple Massey product on degree $i$ associated to $\eta_p$ for some  $3\leq k\leq p^r$.
\end{prop}

Note that when $X$ is the complement of a hyperplane arrangement then, under the hypotheses for which \eqref{iinequality} holds as an equality, it follows that the mod $p$ Betti numbers $b_i(X_r,\F_p)$ of the cover $X_r$ are combinatorially determined, extending the above mentioned result of Yoshinaga to arbitrary primes. 

On the other hand, if there exist non-trivial higher order Massey products on $X$, and assuming that $H_*(X,\Z)$ has no $p$-torsion, we obtain sharper bounds for the 
Betti numbers of rank-one local systems than the ones 
obtained in \cite{PS10}. Specifically, we  show the following.
\begin{theorem}\label{thm inequality} 
For $\nu\colon \pi_1(X)\twoheadrightarrow \Z$  as before, fix  $\lambda\in \C^*$ of prime order $p$, and denote by $L_\lambda$ the rank one $\C$-local system on $X$ defined by sending the generator of $\Z$ to $\lambda$.  Then for any integer $i$ we have the inequality
\begin{equation}\label{iineq} b_i(X,L_\lambda)\leq \beta_i(X,\eta_p)+\dfrac{b_i(X,\F_p)-b_i(X,\C)}{p-1}.\end{equation}
If, moreover, $H_*(X,\Z)$  has no $p$-torsion, then 
\begin{equation} \label{iinequality PS}
   b_i(X,L_\lambda)\leq \beta_i(X,\eta_p). 
\end{equation} 
If we further assume that some $k$-tuple Massey product on degree $j$ associated to $\eta_p$ is non-trivial for some  $3\leq k\leq p$, then 
\begin{equation} \label{iinequality sharper pS}
   b_j(X,L_\lambda)< \beta_j(X,\eta_p). 
\end{equation} 
 \end{theorem}
\begin{rmk} The inequality (\ref{iinequality PS}) was proved more generally, for $\lambda$ of order $p^r$ with $r\geq 1$, by Papadima--Suciu in \cite[Theorem C]{PS10} by a different method. 
\end{rmk}

\subsection{Hyperplane arranegements}
In Section \ref{sec5}, we specialize the above results to the case of hyperplane arrangement complements. 
It is well-known that the complement of a hyperplane arrangement is formal over $\C$ in the sense of Sullivan’s rational
homotopy theory, but in general it is not formal over $\F_p$.  In this regard, we recall Matei's example of an arrangement complement with non-trivial triple Massey products over $\F_p$ \cite{Mat} and indicate how Theorem \ref{thm inequality} applies in this example. On the other hand, results of Cirici--Horel \cite{CH22} provide many examples of hyperplane arrangement complements having all higher order Massey products trivial over $\F_p$. In particular, as discussed in Proposition \ref{cor2i}, the mod $p$ Betti number  $b_i(X_r,\F_p)$ of the cover $X_r$ are in this case combinatorially determined. As a concrete example, we discuss in detail the case of graphic arrangements.

\medskip

{\bf Acknowledgments.} 
The authors would like to thank Wentao Xie for the computations in Example \ref{example Matei}.
Y. Liu is supported by the Project of Stable Support for Youth Team in Basic Research Field, CAS (YSBR-001), NSFC grant No. 12571047 and  the starting grant from University of Science and Technology of China. L. Maxim acknowledges support from the Simons Foundation and from the project ``Singularities and Applications'' - CF 132/31.07.2023 funded by the European Union - NextGenerationEU - through Romania's National Recovery and Resilience Plan.

\section{Alexander modules and the cochain algebra}\label{sec2}

Let $X$ be a connected finite CW complex with a group epimorphism $\nu \colon \pi_1(X)\twoheadrightarrow \Z$.
Consider the corresponding infinite cyclic cover $X^{\nu}$.  Fix a ground field $\K$, and set $R=\K[\Z]\simeq \K[t^{\pm 1}]$.  Then $H_i(X^\nu, \K)$ is a finitely generated $R$-module for any integer $i$.

Consider the local system $\sL^\nu $ on $X$ with stalk $R$, and representation of the fundamental group defined by the composition: 
$$ \pi_{1}(X)  \xrightarrow{\nu} \Z   \rightarrow   Aut(R),$$
with the second map being given by  $1_{\Z}\mapsto t$.  Here $t$ is the automorphism of $R$ given by multiplication by $t$. Then we have the $R$-module isomorphism $H_i(X,\sL^\nu) \cong H_i(X^\nu,\K)$ for any integer $i$.

In this section, we show that  the completion of the $R$-module $H^i(X, \sL^\nu)$  at the maximal
ideal $(t-1)$  is determined by a cochain complex $C^*(X,\K)$ as differential graded algebra (short as dg-algebra).  In particular, by using the universal coefficient theorem, it follows that the eigenvalue $1$ part of $\Tors_R H_i(X^\nu,\K)$ is determined by $C^*(X,\K)$  and the group homomorphism $\nu$, where we use $\Tors_R$ to denote the torsion part of an $R$-module. 

\begin{remark}
We will use simplicial cochain complexes. Let us first explain our convention for the cap product and clarify various sign conventions. 
Given a $k$-simplex $\langle v_0, \ldots, v_k\rangle$ and a $p$-cochain $\alpha$, the cap product is defined by
\[
\langle v_0, \ldots, v_k\rangle\cap \alpha=\alpha(\langle v_0, \ldots, v_p\rangle) \langle v_p, \ldots, v_k\rangle.
\]
Thus, in our convention, in a cap product the chain is always on the left and the cochain is on the right. Moreover, the associativity of the cap product is of the form
\[
(\omega\cap \alpha)\cap \beta=\omega\cap(\alpha\cup \beta).
\]
The Leibniz rule takes the form 
\[
\partial(\omega\cap \alpha)=(-1)^{|\alpha|}\partial\omega\cap\alpha-(-1)^{|\alpha|}\omega\cap\delta\alpha , 
\]
where $|\alpha|$ denotes the degree of $\alpha$. In fact, it can be derived from the definition by using the Leibniz rule for cup product as follows
\begin{align*}
&\partial(\omega\cap \alpha) \cap \beta=(\omega\cap \alpha) \cap \delta\beta=\omega\cap (\alpha\cup \delta\beta)\\
=&\omega\cap\big((-1)^{|\alpha|}\delta(\alpha\cup \beta)-(-1)^{|\alpha|}\delta\alpha\cup \beta\big)=(-1)^{|\alpha|}\omega\cap \delta(\alpha\cup \beta)-(-1)^{|\alpha|}\omega\cap (\delta\alpha\cup \beta)\\
=&(-1)^{|\alpha|}\partial\omega\cap (\alpha\cup \beta)-(-1)^{|\alpha|}(\omega\cap\delta\alpha)\cap \beta=(-1)^{|\alpha|}(\partial\omega\cap\alpha)\cap \beta-(-1)^{|\alpha|}(\omega\cap\delta\alpha)\cap \beta.
\end{align*}
\end{remark}

\medskip

The following construction and notations will be used in the formulation of the main results of this section below.

Since $\nu \colon \pi_1(X)\to \Z$ can be regarded as an element of $H^1(X,\Z)$ and $S^1=\R/\Z=K(\Z,1)$, 
there is a continuous map $f_\nu\colon X\to S^1$ whose induced homomorphism on fundamental groups is $\nu$. Let $q\colon \R\to S^1$ be the covering map, and consider the following Cartesian diagram
\begin{equation}
\xymatrix{
X^\nu\ar_{p}[d]\ar^{f'_\nu}[r]&\R\ar^{q}[d]\\
X\ar^{f_\nu}[r]&S^1.
}
\end{equation}

Since every finite CW complex is homotopy equivalent to a finite simplicial complex, by the simplicial approximation theorem we can assume that $X$ is a finite simplicial complex and $f_\nu$ is a map of simplicial complexes. Without loss of generality, we assume that the simplicial structure on $S^1=\R/\Z$ is given by vertices $\{0=1, \frac{1}{N}, \ldots, \frac{N-1}{N}\}$ ($N\geq 2$) and edges $\{[0, \frac{1}{N}], \ldots, [\frac{N-1}{N}, 1]\}$. For $0\leq i\leq N-1$, denote the vertex $\frac{i}{N}$ by $x_i$, and denote the edge $[\frac{i}{N}, \frac{i+1}{N}]$ by $\gamma_{i+1}$. As covering spaces, we endow $\R$ and $X^\nu$ with the induced simplicial structures from $S^1$ and $X$, respectively. For the remaining part of the proof, we will only work with simplicial chains and cochains. 

Let $\eta_0$ be the $1$-cochain on $S^1$ defined by 
\begin{equation}\label{circle}
\eta_0(\gamma_i)=
\begin{cases}
0& \text{when } i\neq N,\\
1& \text{when } i= N.
\end{cases}
\end{equation}
Since $S^1$ is one-dimensional, $\eta_0$ is also a 1-cocycle. Let 
\begin{equation}\label{cocycle} \eta=f_\nu^*(\eta_0) \ \text{ and } \ \eta'=p^*(\eta) \end{equation} 
be the induced $1$-cocyles on $X$ and $X^\nu$, respectively. 

\medskip

Let us next define, for any $m\in \Z_{>0}$,  a twisted chain complex of $R/(s^{m+1})$-modules, denoted $C_*(X, \eta, m)$, by
\[
\left(C_*(X)\otimes_{\K} R/(s^{m+1}), \partial+s\eta\right),
\]
where $s=t-1$, $(C_*(X), \partial)$ is the simplicial chain complex of $X$, and with boundary map defined by \[(\partial+s\eta)(\Delta)=\partial \Delta+s(\Delta\cap \eta), \] for a simplex $\Delta$ of $X$.
\begin{prop} \label{prop key}
Under the above notations, $(\partial+s\eta)^2=0$. In particular, $C_*(X, \eta, m)$ is a chain complex. Moreover, for any integer $i$ there is an isomorphism of $R/(s^{m+1})$-modules
\begin{equation}\label{eq1}
H_i\big(X, \sL^\nu\otimes_R R/(s^{m+1})\big)\cong H_i\big(C_*(X, \eta, m)\big).
\end{equation}
\end{prop}
\begin{proof}
We prove the first statement by explicit computation. For a simplicial chain $\omega$, we have
\begin{align*}
(\partial+s \eta)^2\omega&=\partial^2\omega+s\partial(\omega\cap \eta)+s \partial\omega\cap \eta+s^2\omega\cap(\eta\cup\eta)\\
&=\partial^2\omega-s\partial\omega\cap\eta+s\omega\cap\delta\eta+s\partial\omega\cap\eta+s^2\omega\cap(\eta\cup\eta)\\
&=\partial^2\omega+s\omega\cap\delta\eta+s^2\omega\cap(\eta\cup\eta)\\
&=0.
\end{align*}
Here we use the fact that $\eta$ is a $1$-cocycle on $X$, and also $\eta \cup \eta=0$ since $\eta$ is the pull-back of a $1$-cocycle on $S^1$ (cf. \eqref{cocycle}).

For the second part, we construct an isomorphism of complexes of $R/(s^{m+1})$-modules:
\begin{equation}\label{eq3}
\Phi: \left(C_*(X)\otimes_{\K}R/(s^{m+1}), \partial+s\eta \right)\to \left(C_*(X^\nu)\otimes_{R} R/(s^{m+1}), \partial_\nu\right),
\end{equation}
where $\partial$ and $\partial_\nu$ are the boundary maps of the simplicial chain complexes $C_*(X)$ and $C_*(X_\nu)$, respectively.

We first define $\Phi$ as a graded $R/(s^{m+1})$-module isomorphism, and then check its compatibility with the boundary maps. Let $\Delta\hookrightarrow X$ be a simplex of $X$. Since $f'_\nu$ is a map of simplicial complexes, there is a unique lift $\Delta'\hookrightarrow X^\nu$ such that $f'_\nu(\Delta'^\circ)\subset [0, 1)$, where $\Delta'^\circ$ denotes the interior of $\Delta'$. 
In particular, for a vertex $v$ of $X$, $v'\in X^\nu$ is the lifting such that $f'_\nu(v')\in [0, \frac{N-1}{N}]$.
Let $\Phi(\Delta\otimes 1)=\Delta'\otimes 1$. Notice that the set of $\Delta'$ defined above form a $R$-module basis of $C_*(X^\nu)$. Hence $\Phi$ is a graded $R/(s^{m+1})$-module isomorphism. 

Next, we prove that $\Phi$ is compatible with the boundary maps. We need to check that
\[
\partial_\nu\Phi(\Delta)=\Phi\big((\partial+s\eta)\Delta\big),
\]
where $\Delta$ is a $k$-simplex of $X$. 
If $f_\nu(\Delta)\neq [\frac{N-1}{N}, 1]$, then $\Delta\cap \eta=0$ and the above equality follows from the functoriality of chain complexes. 
So we only need to consider the case when $f_\nu(\Delta)=[\frac{N-1}{N}, 1]$. 

We choose the total ordering of vertices of $S^1$ by $x_0<x_1<\cdots <x_{N-1}$, and choose a total ordering of vertices of $X$, such that $f\colon X\to S^1$ preserves the ordering of the vertices. Furthermore, choose a total ordering of vertices of $X^\nu$ such that $p\colon X^\nu\to X$ preserves the ordering of the vertices. Write $\Delta=\langle v_0, \ldots, v_k\rangle$ according to the chosen ordering of vertices of $X$. 
 
Since $f_\nu$ preserves the order of the vertices of $X$ and $S^1$, and moreover $f_\nu(\Delta)=[\frac{N-1}{N}, 1]$, there exists $j\in \{0, \ldots, k-1\}$ such that ,
without any loss of generality, we may assume that
\[
f(v_0)=\cdots=f(v_j)=\frac{N-1}{N}\quad \text{and}\quad f(v_{j+1})=\cdots=f(v_k)=1.
\]
Consider the case when $j\geq 1$. In this case, 
\[
\Delta\cap \eta=(\langle v_0, v_1\rangle\cap \eta)\langle v_1, \dots, v_k\rangle=\left(\left\langle \frac{N-1}{N}, \frac{N-1}{N}\right\rangle\cap \eta_0\right)\langle v_1, \dots, v_k\rangle=0.
\]
Denote the action of $1\in \Z$ on $X^{\nu}$ by $\sigma$. Then
\begin{align*}
&\Phi\big((\partial+s\eta)\Delta\big)=\Phi(\partial\Delta)\\
=&\sum_{0\leq i\leq j}(-1)^i\Phi\big(\langle v_0, \ldots, \widehat{v_i}, \ldots, v_j, \ldots v_k\rangle\big)+\sum_{j+1\leq i\leq k}(-1)^i\Phi\big(\langle v_0, \ldots, v_j, \ldots, \widehat{v_i}, \ldots, v_k\rangle\big)\\
=&\sum_{0\leq i\leq j}(-1)^i\left\langle v'_0, \ldots, \widehat{v'_i}, \ldots, v'_j, \sigma(v'_{j+1}), \ldots \sigma(v'_k)\right\rangle\\
&+\sum_{j+1\leq i\leq k}(-1)^i\left\langle v'_0, \ldots, v'_j, \sigma(v'_{j+1}), \ldots, \widehat{\sigma(v'_i)}, \ldots, \sigma(v'_k)\right\rangle\\
=&\partial_\nu\langle v'_0, \ldots, v'_j, \sigma(v'_{j+1}), \ldots \sigma(v'_k)\rangle\\
=&\partial_\nu\Phi(\Delta).
\end{align*}

Consider the case when $j=0$. In this case, $\Delta\cap \eta=\langle v_1, \ldots, v_k\rangle$. Hence,
\begin{align*}
\Phi((\partial+s\eta)\Delta)=&\sum_{1\leq i\leq k}(-1)^i\Phi\big(\langle v_0, v_1, \ldots, \widehat{v_i}, \ldots v_k\rangle\big)+\Phi\big(\langle v_1, \ldots, v_k\rangle\big)+s\Phi\big(\langle v_1, \ldots, v_k\rangle\big)\\
=&\sum_{1\leq i\leq k}(-1)^i\Phi\big(\langle v_0, v_1, \ldots, \widehat{v_i}, \ldots v_k\rangle\big)+t\Phi\big(\langle v_1, \ldots, v_k\rangle\big)\\
=&\sum_{1\leq i\leq k}(-1)^i\left(\Big\langle v'_0, \sigma(v'_1), \ldots, \widehat{\sigma(v'_i)}, \ldots \sigma(v'_k)\Big\rangle\right)+t\big(\langle v'_1, \ldots, v'_k\rangle\big)\\
=&\sum_{1\leq i\leq k}(-1)^i\left(\left\langle v'_0, \sigma(v'_1), \ldots, \widehat{\sigma(v'_i)}, \ldots \sigma(v'_k)\right\rangle\right)+\big(\big\langle \sigma(v'_1), \ldots, \sigma(v'_k)\big\rangle\big)\\
=&\partial_\nu \big\langle v'_0, \sigma(v'_1), \ldots,  \sigma(v'_k)\big\rangle\\
=&\partial_\nu\Phi(\Delta).\qedhere
\end{align*}
\end{proof}

Let $C^*(X)$  be the simplicial cochain complex of $X$. Noting that $C_*(X^\nu)$ is a complex of free $R$-modules, we denote by $C_R^*(X^\nu)$ its dual complex of $R$-modules. Given $m\in \Z_{>0}$, taking duals (as complexes of free $R/(s^{m+1})$-modules) of \eqref{eq3}, we have an isomorphism of complexes 
\begin{equation*}
\Psi: \left(C^*(X)\otimes_{\K}R/(s^{m+1}), \delta +(\eta \cup -)\otimes s \right)\to \left(C_R^*(X^\nu)\otimes_{R} R/(s^{m+1}), \delta_\nu\right),
\end{equation*}
where $\delta$ and $\delta_\nu$ are the differentials (coboundary maps) in $C^*(X)$ and $C_R^*(X^\nu)$, respectively. Since the complex $(C_R^*(X^\nu), \delta_\nu)$ computes the cohomology of $\sL^\nu$, the complex $(C_R^*(X^\nu)\otimes_{R} R/(s^{m+1}), \delta_\nu)$ computes the cohomology of $\sL^\nu\otimes_R R/(s^{m+1})$.  Let 
$$C^*(X, \eta, m):=(C^*(X)\otimes_{\K}R/(s^{m+1}), \delta +(\eta \cup -)\otimes s).$$ Thus, we have the following.

\begin{cor}
For any integer $i\geq 0$, there is an isomorphism of $R/(s^{m+1})$-modules
\begin{equation} \label{eq2}
H^i\big(X, \sL^\nu\otimes_R R/(s^{m+1})\big)\cong H^i\big(C^*(X, \eta, m)\big).
\end{equation}
\end{cor}

Denote the completion of $R$ at the maximal ideal $(t-1)$ by $\widehat{R}$, that is, $$\widehat{R}=\displaystyle\varprojlim_m R/(s^{m+1}).$$ Note that $\widehat{R}= \K[[s]]$ is a principal ideal domain and a regular local ring. There is a natural $\widehat{R}$-isomorphism for any finitely generated $R$-module $A$,
$$ A \otimes_R \widehat{R} \cong \widehat{A},$$ 
where $\widehat{A}$ is the $(t-1)$-adic completion of $A$.

Let
\[
C^*(X, \eta, \infty)=\left(C^*(X)\otimes_{\K} \widehat{R}, \delta +(\eta \cup -)\otimes s\right)
\]
be the complex of $\widehat{R}$-modules, which is also the inverse limit of $C^*(X, \eta, m)$ as $m\to \infty$. 
Since the left side of the isomorphism (\ref{eq2}) is also functorial on $m$, we can take the inverse limit of the cohomology groups. Since the cohomology groups are finite dimensional $\K$-vector spaces, the Mittag-Leffler condition is automatically satisfied. Thus we have the following. 
\begin{cor} \label{hat} As $\widehat{R}$-modules,
$$H^i\big(X,\sL^\nu\otimes_R \widehat{R}\big)\cong H^i\big(C^*(X, \eta, \infty)\big).$$
\end{cor}

Using the universal coefficient theorem, we then get the following.
\begin{cor}\label{cor cochain algebra}  As $\widehat{R}$-modules, the eigenvalue 1 part of 
  $\Tors_{{R}} H_i(X^\nu,\K)$ is isomorphic to
  $$\Tors_{\widehat{R}} H^{i+1}\big(C^*(X, \eta, \infty)\big),$$
where $\Tors_{\widehat{R}}$ means taking the torsion part as $\widehat{R}$-modules. 
\end{cor}

\begin{remark}
When $X$ is a manifold and $\K=\C$, $C^*(X,\K)$ can be replaced by the de Rham complex and all the results in this section are already presented in \cite[Section 3]{BLW} (see also \cite{DP} and \cite{BW15B}). More generally, when $X$ is a finite CW complex and $\mathrm{char}(\K)=0$, we can substitute the de Rham complex by Sullivan's model. So the interesting case for this section is when $\mathrm{char}(\K)>0$.  
\end{remark}

\section{Spectral sequences and Massey products}

The cohomology class $[\eta_0]\in H^1(S^1, \K)$ corresponding to \eqref{circle} is the generator dual to the fundamental class in $H_1(S^1, \K)$. Moreover, $\nu\colon \pi_1(X)\to \Z=\pi_1(S^1)$ induces a homomorphism $\nu^\vee: H^1(S^1, \K)\to H^1(X, \K)$, which is equal to $f_\nu^*$. Denoting $\nu^\vee([\eta_0])$ by $\nu$, we have 
\begin{equation}\label{cochain}
[\eta]=f_\nu^*([\eta_0])=\nu^\vee[\eta_0]=\nu\in H^1(X, \K).
\end{equation}
Since $[\eta]$ is the pullback of a cohomology class in $H^1(S^1,\K)$, we note that  $[\eta] \cup [\eta]=0$. 
Then we get the following Aomoto complex associated to $[\eta]$ by (left) cup product
$$
(H^{\ast}(X,\K), [\eta]\cdot)\colon \cdots \to H^{i-1}(X,\K) \xrightarrow{[\eta]\cup}  H^{i}(X,\K)\xrightarrow{[\eta]\cup}  H^{i+1}(X,\K)\to \cdots.
$$

\bd 
With the above notations, we define {\it  the $i$-th Aomoto Betti number with $\K$-coefficients} as
\[ \beta^\K_{i}(X,\eta)\coloneqq\dim_{\K}H^{i}(H^{\ast}(X,\K), [\eta]\cup -). \]
\ed

As suggested by Corollary \ref{cor cochain algebra},  we will be mainly concerned here with understanding the complex $$( C^*(X,\K)\otimes_\K \widehat{R}, \delta \otimes \mathrm{id} +  (\eta \cup-) \otimes s).$$ 
This complex can be viewed as a double complex 
$$   A^{i,j}\coloneqq C^{i+j}(X,\K)\otimes \K s^i \cong C^{i+j}(X,\K)$$
with the vertical map given by $d\otimes \mathrm{id}$ and the horizontal map given  by left cup product with $\eta$. 
Note that $A^{i,j}=0$ for $i+j< 0$ or $i< 0$.  
Then by the stupid filtration, we get a spectral sequence with the first page $$E_1^{i,j}= H^{i+j}(C^*(X,\K))=H^{i+j}(X,\K)$$
with $d_1^{i,j}(\alpha)=[\eta] \cup \alpha$ for any $\alpha\in E_1^{i,j}$.
Consequently,  the second page of the spectral sequence is given by $$ E_2^{i,j}= H^{i+j} \big(H^{\ast}(X,\K),[\eta]\cup - \big).$$
\begin{rmk}
    Let $J=(t-1)$ be the augmention ideal of $R$ generated by $t-1$. The successive powers of $J$ determines a decreasing filtration on $R$, called the $J$-adic filtration. The $J$-adic completion of $R$ can be identified with $\widehat{R}\coloneqq \K[[s]]$ with $s=t-1$. The spectral sequence induced by the $J$-adic filtration has been  studied by Papadima--Suciu in \cite{PS10}, and it is dual to the above spectral sequence by Corollary \ref{cor cochain algebra}.
In fact, the $0$-page of the above spectral sequence is given as follows 
\[
\xymatrixrowsep{8pt}
\xymatrix{
  \vdots &  \vdots &   \vdots &  \vdots  \\
   C^2(X,\K)  \ar[r]^{(-\cup \eta)\cdot s} & C^3(X,\K)\cdot s \ar[r]^{(-\cup \eta)\cdot s}   & C^4(X,\K)\cdot s^2   & \cdots  \\
    C^1(X,\K) \ar[r]^{(-\cup \eta) \cdot s} \ar[u]^{\delta} &  C^2(X,\K)\cdot s \ar[r]^{(-\cup \eta) \cdot s} \ar[u]^{\delta}& C^3(X,\K)\cdot s^2 \ar[u]^{\delta} & \cdots  \\
  C^0(X,\K) \ar[r]^{(-\cup \eta) \cdot s} \ar[u]^{\delta}& C^1(X,\K)\cdot s \ar[r]^{(-\cup \eta) \cdot s} \ar[u]^{\delta}& C^2(X,\K)\cdot s^2 \ar[u]^{\delta}& \cdots\\
 & C^0(X,\K)\cdot s \ar[r]^{(-\cup \eta) \cdot s} \ar[u]^{\delta}  &  C^1(X,\K)\cdot s^2  \ar[u]^{\delta}  & \cdots\\
  & & C^0(X,\K)\cdot s^2 \ar[u]^{\delta} & \cdots \\
  & & & \ddots
},
\]
while the the $0$-page of the spectral sequence induced by the $J$-adic filtration is given by 
\[
\xymatrixrowsep{8pt}
\xymatrix{
 \ddots &\vdots &  \vdots &   \vdots &  \vdots   \\
 &  C_0(X,\K)\cdot s^3 &  C_1(X,\K)\cdot s^2 \ar_{(-\cap \eta)\cdot s}[l] \ar[d]^{\partial} & C_2(X,\K)\cdot s 
\ar_{(-\cap \eta)\cdot s}[l] \ar[d]^{\partial} & C_3(X,\K) \ar_{(-\cap \eta)\cdot s}[l] \ar[d]^{\partial}  \\
 &  &   C_0(X,\K)\cdot s^2 &  C_1(X,\K)\cdot s \ar_{(-\cap \eta)\cdot s}[l] \ar[d]^{\partial} & C_2(X,\K) 
\ar_{(-\cap \eta)\cdot s}[l] \ar[d]^{\partial} \\
 &  & & C_0(X,\K)\cdot s &  C_1(X,\K) \ar_{(-\cap \eta)\cdot s}[l] \ar[d]^{\partial}\\
 & & & &  C_0(X,\K) 
}.
\] 
 \end{rmk}

\smallskip

For the differential maps on the pages $E_k$ with $k\geq 2$, we have to deal with a special kind of $(k+1)$-tuple Massey products, determined by $\eta$. See the survey paper \cite{LM21} for an overview of higher order Massey products.  
Since later on we will use Cirici-Horel's results \cite{CH22} about higher order Massey products, we follow here the terminology from their paper.
\begin{defn}
Let $\omega\in H^i(X,\K)$. 
A defining system for $\{\overbrace{[\eta],\cdots,[\eta]}^{k},\omega \}$ is  a collection of elements of $\{\alpha_1,\cdots,\alpha_k\}$ in $ C^i(X,\K)$ such that for the following $(k+1)\times(k+1)$-matrix
\[
  C=(c_{i,j})= \begin{pmatrix}
    \eta & 0 & \cdots & 0 & 0\\
   0 &  \eta & \cdots & 0 & \alpha_{k}\\
    \vdots & \vdots & \ddots & \vdots& \vdots\\
    0 & 0 & \cdots \ & \eta & \alpha_2\\
    0 & 0 & \cdots & 0 & \alpha_1\\
  \end{pmatrix}  
\]
we have $[\alpha_1]=\omega$
and
$$ \delta(c_{i,j})=\sum_{q=i}^{j-1} (-1)^{\deg c_{i,q}}c_{i,q}c_{q+1,j}.$$ Here $[\alpha_1]$ denotes the cohomology class of the cocycle $\alpha_1$. 
The {\it $(k+1)$-tuple Massey product} for $\{ \overbrace{[\eta],\cdots,[\eta]}^{k},\omega  \}$ is defined to be the set of cohomology classes of $$\sum_{q=1}^k  (-1)^{\deg c_{1,q}}c_{1,q}c_{q+1,k+1}.$$\end{defn}

In our case, the above equations are equivalent (up to signs) to the following
\begin{center}
    $\delta \alpha_1=0, [\alpha_1]=\omega, \delta  \alpha_2=\eta \cup \alpha_1, \cdots, \delta  \alpha_{k}=\eta \cup \alpha_{k-1},$
\end{center} which  in \cite{KP,Paj17,Paj19} is called a $k$-chain starting from $\omega$, 
and the $(k+1)$-tuple Massey product is  $\eta \cup \alpha_k $. 
There is some indeterminacy of this Massey product arising from the choice of $\{\alpha_1,\cdots,\alpha_k\}$ in $ C^i(X,\K)$.
In fact, the indeterminacy  of the $(k+1)$-tuple Massey product of $\eta$ and $\omega$ comes from the corresponding $k$-tuple Massey product of $\eta$ and some $\omega'\in H^{i}(X,\K)$.

\begin{defn}\label{higherM}  Let $\omega\in H^i(X,\K)$ and $k\geq 1$. With the above notations, the image of $\eta \cup \alpha_k$  modulo the indeterminacy is called 
    the {\it $(k+1)$-tuple Massey product associated to $\eta$}, denoted by 
    $$\langle [\eta], \omega\rangle_k=\langle \overbrace{ [\eta],\cdots,[\eta]}^k, \omega \rangle .$$
    If $\langle [\eta], \omega\rangle_k\neq 0$, we say the length of the higher order Massey product {\it on degree $i$ } associated to $\eta$ is  $k+1$.  Note that for $k=2$, this definition yields a special type of classical triple Massey products $ \langle[\eta], [\eta],\omega\rangle \in H^{i+1}(X,\K)/(\eta \cup H^{i}(X,\K))$.
\end{defn} 
\begin{proof}[Proof of Theorem \ref{thm main}]
    By standard computations with spectral sequences (see, e.g., \cite[pages 163--165]{BT}), one gets that
the differential maps $d_k$ for the above-mentioned spectral sequence are computed  exactly by the $(k+1)$-tuple Massey products associated to $\eta$.  
Since $\widehat{R}= \K[[s]]$ is a principal ideal domain, any bounded complex of  finitely generated $\widehat{R}$-modules is quasi-isomorphic to a finite direct sum of complexes of type $0\to \widehat{R}\to 0$ or $0\to \widehat{R}\overset{\cdot s^j}{\to} \widehat{R}  \to 0$ for some $j\geq 0$ at various degrees. It is easy to see that the spectral sequence for the second type degenerates on the $(j+1)$-th page, while the module $ \widehat{R}/(s^j)$ has the Jordan block size equal to $j$. 
So the maximal size of Jordan blocks for the eigenvalue $1$ part of $H_i(X^\nu,\K)$ is one less than the number $\min_k \{d_k^{p,q}=0 \text{ for all } p+q=i\}$, which by definition is exactly the  length of the highest nonvanishing  Massey product  on degree $i$ associated to  $\eta$.  
Then the assertion of Theorem \ref{thm main} follows. (One may also check \cite[Proposition 9.1]{PS10} for more details.)
\end{proof}

\begin{proof}[Proof of Corollary \ref{cor singular}]
For complex coefficients, one can use Sullivan's de Rham cdga to replace $C^*(X,\C)$. Note that for {a} complex algebraic variety,  Sullivan's de Rham cdga is equivalent to a cdga admitting an extra grading inducing the weight filtration on $H^*(X,\C)$, see \cite[Theorem 4.2]{BR} and \cite[Theorem 8.11, Remark 8.12]{CH20}. In particular, the differential map  preserves the weight filtration, see \cite[Theorem 4.2(ii)]{BR}. 
Note that the weight filtration on $H^i(X,\C)$ has range $[0,\min\{2i,2n\}]$. For any $\alpha \in H^i(X,\C)$,  
   the assumption $W_0 H^1(X,\C)=0$ implies that $d_k(\alpha)$ increase the weight at least by $k$. 
    Hence $d_k(\alpha)=0$ if $k> \min\{2i+2,2n\}$. 
Same proof applies to the case $W_1 H^1(X,\C)=0$.
\end{proof}

\section{Prime-power cyclic covers. Homology of local systems. Betti bounds}\label{sec4}
In this section, we assume that $X$ is a connected finite CW complex with a group epimorphism $\nu \colon \pi_1(X)\twoheadrightarrow \Z$. If $p$ is a prime integer and $r>0$, by further projecting to $\Z_{p^r}$ we get a corresponding  $p^r$-fold cover $X_r$ of $X$. 
Let $\K=\F_p$ and $\eta_p \in C^1(X,\F_p)$ be the corresponding $1$-cocycle (where the subscript is used to emphasize the characteristic of the field).
Note that  $$\F_p \Z_{p^r}=\F_p[t]/(t^{p^r}-1)=\F_p[t]/((t-1)^{p^r}).$$ Then by the $J$-adic filtration, one gets a truncated spectral sequence induced from the one in the previous section, which computes the homology $H_*(X_r,\F_p)$. 

By Proposition \ref{prop key},  this spectral sequence has $p^r$ columns, hence it always degenerates at most at the $E_{p^r}$-page. Moreover,  the differential map on the $E_1$-page of the spectral sequence 
is given by the left cup product with $\eta_p$, 
and 
 the differential map $d_k$ on the $E_k$-page for $k\geq 2$ is computed  exactly by the $(k+1)$-tuple Massey products associated to $\eta_p$, since this spectral sequence is a truncated version of the one used in the previous section. Then Proposition \ref{cor2i} follows directly. 
Let us next prove Proposition \ref{prop prime tower cover} and Theorem \ref{thm inequality}.
\begin{proof}[Proof of Proposition \ref{prop prime tower cover}]
By the $E_2$-page of the spectral sequence, we have
\begin{align*}
          b_i(X_r,\F_p) \leq & \sum_{k=i-p^r+1}^i \dim E_2^{i-k,k} \\
                   =   &  \dim \ker(d_2^{0,i})+\dim \coker(d_2^{p^r-1,i-p^r+1})+ (p^r-2)\cdot   \beta_i(X,\eta_p)\\
                   =   & b_i(X,\F_p) + (p^r-1)\cdot   \beta_i(X,\eta_p)
\end{align*}
Here the second equality follows from the fact $$\ker(d_2^{0,i})=\ker \{ H^i(X,\F_p) \xrightarrow{\eta \cup} H^{i+1}(X,\F_p)\} $$
and $$ \coker(d_2^{p^j-1,i-p^j+1})=\coker \{ H^{i-1}(X,\F_p) \xrightarrow{\eta \cup} H^{i}(X,\F_p)\}.$$
\end{proof}

 \begin{proof}[Proof of Theorem \ref{thm inequality}] Let $Y$ denote the $p$-fold cover of $X$ associated to $\nu\colon \pi_1(X)\to \Z_p$.
Since $p$ is a prime number, we have
 $$b_i(Y,\C)= b_i(X,\C)+(p-1)b_i(X,L_\lambda). $$
On the other hand, we have by the universal coefficient theorem and \eqref{iinequality} that
$$ b_i(Y,\C)                  \leq  b_i(Y,\F_p) \leq  b_i(X,\F_p)+   (p-1)\cdot \beta_i(X,\eta_p). $$ 
Then the  inequality \eqref{iineq} follows. The  inequality \eqref{iinequality PS} is obvious. 

If a $k$-tuple Massey product on degree $j$ associated to $\eta_p$ is non-trivial for some  $k\leq p$, then we have by Proposition \ref{cor2i} that $$b_j(Y,\F_p) <  b_j(X,\F_p)+   (p-1)\cdot \beta_j(X,\eta_p),$$ so the claim follows.  \end{proof}

\section{Hyperplane arrangements}\label{sec5}
In this section, we specialize to the case of hyperplane arrangement complements. For relevant background on hyperplane arrangements, we refer to Dimca's book \cite{Dim17}. 

Let $\sA$ be an arrangement of $d$ hyperplanes in $\C^n$ with complement $X$.
Note that $H_1(X,\Z)\cong \Z^d$ is torsion free. Then any epimorphism $\pi_1(X)\twoheadrightarrow \Z_{p^r}$ must factor through some group epimorphism $\nu \colon \pi_1(X)\twoheadrightarrow \Z$. Hence Proposition \ref{prop prime tower cover}, Proposition  \ref{cor2i} and Theorem \ref{thm inequality} apply to the case of hyperplane arrangement complements.

It is well-known that the complement of a hyperplane arrangement is formal over $\C$ in the sense of Sullivan’s rational
homotopy theory, see, e.g., \cite{Dup} for a discussion and references, but in general it is not formal over $\F_p$. 
We recall here the following example given by Matei \cite{Mat}.

\begin{ex} \label{example Matei}
Let $p>2$ be a prime integer. Consider the full monomial arrangement $\sA(p,1,3)$ in $\C^3$ defined by the zero locus of
$$ z_1\cdot z_2 \cdot z_3 \cdot\prod_{1\leq i<j\leq 3} (z_i^p-z_j^p).$$
We label the hyperplanes as follows
\begin{center}
    $H_{i,j}^q=\{ z_i-\lambda^q z_j\}$, where $\lambda=\exp(2\pi i/p)$ and $1\leq i<j\leq 3$ and $0\leq q \leq p-1$,
\end{center}
\begin{center}
   $H_{3p+i}=\{z_i=0\} $ for $1\leq i \leq 3$. 
\end{center}
Matei showed \cite{Mat} that the complement $X$ of  $\sA(p,1,3)$ has non-trivial  triple Massey products on $H^2(X,\F_p)$. More precisely,  for
    \begin{center}
        $\nu=(\overbrace{1, \cdots, 1}^{p}, \overbrace{-1, \cdots, -1}^{p},\overbrace{0, \cdots, 0}^{p},0,0,0)$ and $\omega=(\overbrace{0, \cdots, 0}^{p},\overbrace{1, \cdots, 1}^{p}, \overbrace{-1, \cdots, -1}^{p},0,0,0)$
    \end{center}
  he showed that $\langle\nu,\nu,\omega\rangle\neq 0$.   

Now we focus on the arrangement  $\sA(3,1,3)$, i.e., $p=3$. 
  Let $L_\lambda$ be rank one $\C$-local system on $X$ corresponding to $$(\lambda, \lambda, \lambda, \lambda^{-1}, \lambda^{-1}, \lambda^{-1},\overbrace{1, \cdots, 1}^{6})\in (\C^*)^{6} $$
  with $\lambda=\exp(2\pi i/3)$.
A presentation for the fundamental group of $X$ can be found in \cite[Section 3.2]{Mat}.
 Then a computation by Fox calculus (using a computer) yields that
  $$1=b_1(X,L_\lambda)< \beta_1(X,\eta_3)=2. $$
  This is compatible with the strict inequality (\ref{iinequality sharper pS}).
\end{ex}

On the other hand, the following result due to Cirici--Horel (stated here in the context of hyperplane arrangements) provides many examples of hyperplane arrangement complements having trivial higher order Massey products over $\F_p$.

\begin{theorem}\cite[Theorem 7.12]{CH22,CH24}
    Let $X$ be the complement of a hyperplane arrangement $\sA$ in $\C^n$. 
    Assume that $\sA$ is defined over $K$, where $K$ is a $\ell$-adic field (i.e., a finite extension of $\Q_\ell$) and let $q =\ell^m$ be the cardinality of its residue field. Here $\ell$ is a prime number different from $p$. Write $h$ for the order of $q$ in the group $\F_p^*$. If $(k-2)/h\notin \Z$, then all $k$-tuple Massey products are trivial in $H^*(X,\F_p)$. \end{theorem}

\br Cirici-Horel use the singular cochain algebra in \cite{CH22}, while we work with the simplicial cochain algebra.   Since every finite CW complex is homotopy equivalent to a finite simplicial complex, \cite[Theorem 7.12]{CH22,CH24} remains valid in our setting. 
\er 

\bex \label{ex graphic}
Let $\sA$ be a graphic arrangement. For any prime number $p>2$, we can choose $q=\ell=2 $. 
In particular, if $p=3$, we have $h=2$; if $p=5$, we have  $h=4$.
Hence the inequality (\ref{iinequality}) becomes an equality
$$b_i(Y,\F_p)=b_i(X,\F_p) + (p-1)\cdot \beta_i(X,\eta_p),$$
which is determined by the combinatorics with $Y$ some $p$-fold cover of $X$ for $p=3$ or $5$.
\eex

The above examples raise the question of combinatorial invariance of the higher order Massey products of hyperplane arrangement complements over $\F_p$. This problem seems to have been addressed to some extent in an unpublished work of Rybnikov \cite{Ry}.

\medskip

We end this paper with the following remark.
\begin{rmk} 
Let $X$  be a hyperplane arrangement complement. Using (\ref{mod 2 equation}), 
Yoshinaga showed that the Betti number $b_i(Y,\F_2)$ for the double cover $Y$ is determined by combinatorics  \cite{Yos}.
He also asked the following question \cite[Remark 3.8]{Yos}: 
\begin{center}
    Is the cohomology ring structure of $H^*(Y,\F_2)$ combinatorially determined?
\end{center}

Let $Y'$ be a double cover of $Y$, hence a $4$-fold cover of $X$. By (\ref{mod 2 equation}), $b_i(Y',\F_2)$ is determined by cohomology ring structure of $H^*(Y,\F_2)$. On the other hand, Proposition \ref{prop prime tower cover} shows that $b_i(Y',\F_2)$ can be computed using the information from the higher order Massey products. 
This suggests that there is a  subtle connection between the cohomology ring structure of $H^*(Y,\F_2)$ and the higher order Massey products. 
\end{rmk}

\bibliographystyle{amsalpha}

\begin{thebibliography}{AD}
\bibitem{BFM} G. Bazzoni, M. Fern\'andez, V. Mu\~noz, {\it Non-formal co-symplectic manifolds}, Trans. Amer. Math. Soc. 367 (2015), no.6, 4459--4481.

\bibitem{BT} R. Bott, L. Tu, {\it Differential forms in algebraic topology}, Graduate Texts in Mathematics, 82. Springer-Verlag, New York-Berlin, 1982.

\bibitem{BLW} N. Budur, Y. Liu, B. Wang, {\it The monodromy theorem for compact K\"ahler manifolds and smooth quasi-projective varieties},  Math. Ann. 371 (2018), no. 3-4, 1069--1086.

\bibitem{BR} N. Budur, M. Rubi\'o, {\it L-infinity pairs and applications to singularities}, Adv. Math.354 (2019), 106754, 40 pp.




\bibitem{BW15B} N. Budur, B. Wang, {\it Cohomology jump loci of differential graded Lie algebras,}  Compos. Math. 151 (2015), no. 8, 1499--1528.

\bibitem{CH20} J. Cirici, G. Horel, {\it Mixed Hodge structures and formality of symmetric monoidal functors}, Ann. Sci. \'Ec. Norm. Sup\'er. (4) 53 (2020), no.4, 1071--1104.

\bibitem{CH22} J. Cirici, G. Horel, {\it \'Etale cohomology, purity and formality with torsion coefficients}, J. Topol. 15 (2022), no.4, 2270--2297.

\bibitem{CH24} J. Cirici, G. Horel, {\it Corrigendum: \'Etale cohomology, purity and formality with torsion coefficients}, J. Topol. 17 (2024), no. 2, Paper No. e12348, 4 pp.



\bibitem{DGMS} P. Deligne, P. Griffiths, J. Morgan, D. Sullivan, {\it Real homotopy theory of K\"ahler manifolds},  Invent. Math. 29 (1975), no. 3, 245--274.






\bibitem{Dim17} A. Dimca, {\it Hyperplane arrangements. An introduction}, Universitext. Springer, Cham, 2017.  






\bibitem{DP} A. Dimca, S. Papadima, {\it Non-abelian cohomology jump loci from an analytic viewpoint}, Commun. Contemp. Math. 16 (2014), no. 4, 1350025.

\bibitem{Dup} C. Dupont, {\it  Purity, formality, and arrangement complements,} 
Int. Math. Res. Not. IMRN 2016, no. 13, 4132--4144.

\bibitem{Eke} T. Ekedahl, {\it Two examples of smooth projective varieties with nonzero Massey products}, Algebra, algebraic topology and their interactions (Stockholm, 1983), 128--132. Lecture Notes in Math., 1183. Springer-Verlag, Berlin, 1986.

\bibitem{Memoir}
E. Elduque, C. Geske, M. Herrad\'on Cueto, L. Maxim, B. Wang, {\it Mixed Hodge structures on Alexander modules}, Mem. Amer. Math. Soc. 296 (2024), no. 1479.




\bibitem{FGM} M. Fern\'andez, A. Gray, J. Morgan, {\it  Compact symplectic manifolds with free circle actions, and Massey products}, Michigan Math. J. 38 (1991), no.2, 271--283.



\bibitem{KP} T. Kohno, A. Pajitnov,  {\it Novikov homology, jump loci and Massey products,} Cent. Eur. J. Math. 12 (2014), no.9, 1285-1304.


\bibitem{LM21} I. Limonchenko, D. Millionshchikov, {\it Higher order Massey products and applications}, Topology, geometry, and dynamics--V. A. Rokhlin-Memorial, 209--240. Contemp. Math., 772, American Mathematical Society, Providence, RI, 2021. 

\bibitem{Mat} D. Matei, {\it Massey products of complex hypersurface complements}, Singularity theory and its applications, Advanced Studies in Pure Mathematics, vol. 43, 205--219, Math. Soc. Japan, Tokyo, 2006.



\bibitem{Paj17} A. Pajitnov, {\it Massey products in mapping tori}, Eur. J. Math. 3 (2017), no.1, 34--42.


\bibitem{Paj19} A. Pajitnov, {\it  Twisted monodromy homomorphisms and Massey products,} Topology Appl. 255 (2019), 15--31.


\bibitem{PS10} S. Papadima, A. I. Suciu, {\it The spectral sequence of an equivariant chain complex and homology with local coefficients}, Trans. Amer. Math. Soc. 362 (2010), no. 5, 2685--2721. 

\bibitem{PS10B} S. Papadima, A. I. Suciu, {\it Algebraic monodromy and obstructions to formality.} Forum Math. 22 (2010), no.5, 973-983.






\bibitem{Rez} A. Reznikov, {\it Three-manifolds class field theory (homology of coverings for a nonvirtually $b_1$-positive manifold)}, Selecta Math. (N.S.) 3 (1997), no. 3, 361--399.

\bibitem{Ry}  G. L. Rybnikov, {\it  On the fundamental group and triple Massey's product}, arXiv:math/9805061.







\bibitem{Yos} M. Yoshinaga, {\it  Double coverings of arrangement complements and 2-torsion in Milnor fiber homology}, Eur. J. Math. 6 (2020), no. 3, 1097--1109.


\end{thebibliography}

\end{document}